\documentclass{amsart}
\usepackage{latexsym,amssymb,enumerate, amsmath}
\usepackage{stmaryrd}
\usepackage{wasysym}

\newtheorem{theorem}{Theorem}[section]

\newtheorem{lemma}[theorem]{Lemma}

\theoremstyle{definition}

\newtheorem*{thmA}{Theorem A}
\newtheorem*{Coro}{Corollary B}

\newenvironment{enumeratei}{\begin{enumerate}[\upshape (a)]}
    {\end{enumerate}}

\def\irr#1{{\rm Irr}(#1)}
\def\cent#1#2{{\bf C}_{#1}(#2)}

\def\syl#1#2{{\rm Syl}_#1(#2)}

\def\nor{\trianglelefteq\,}
\def\oh#1#2{{\bf O}_{#1}(#2)}
\def\zent#1{{\bf Z}(#1)}
\def\sbs{\subseteq}

\def\fit#1{{\bf F}(#1)}
\def\frat#1{{\bf \Phi}(#1)}

\def\frat#1{{\bf \Phi}(#1)}
\def\fitd#1{{\bf F}_{2}(#1)}
\def\irr#1{{\rm Irr}(#1)}

\def\cent#1#2{{\bf C}_{#1}(#2)}

\def\syl#1#2{{\rm Syl}_#1(#2)}
\def\nor{\trianglelefteq\,}

\def\oh#1#2{{\bf O}_{#1}(#2)}
\def\Oh#1#2{{\bf O}^{#1}(#2)}
\def\zent#1{{\bf Z}(#1)}
\def\sbs{\subseteq}

\def\alt#1{{\rm Alt}(#1)}

\def\fit#1{{\bf F}(#1)}
\def\GF#1{{\rm GF}(#1)}

\def\V#1{{\rm V}(#1)}
\def\E#1{{\rm E}(#1)}

\def\irr#1{{\rm Irr}(#1)}

\def\cd#1{{\rm cd}(#1)}

\def\cent#1#2{{\bf C}_{#1}(#2)}
\def\syl#1#2{{\rm Syl}_#1(#2)}
\def\oh#1#2{{\bf O}_{#1}(#2)}
\def\Oh#1#2{{\bf O}^{#1}(#2)}
\def\zent#1{{\bf Z}(#1)}

\def\ker#1{{\rm ker}(#1)}

\def\alt#1{{\rm Alt}(#1)}

\def\sbs{\subseteq}

\mathchardef\coso="2023
\def\ww#1{{#1}^{\coso}}
\def\cE{\bar{\rm E}}

\begin{document}

\title[Character degree graph of solvable groups]{On the character degree graph \\of solvable groups}

\author[Z. Akhlaghi et al.]{Zeinab Akhlaghi}
\address{Zeinab Akhlaghi, Faculty of Math. and Computer Sci., \newline Amirkabir University of Technology (Tehran Polytechnic), 15914 Tehran, Iran.}
\email{z.akhlaghi@aut.ac.ir}

\author[]{Carlo Casolo}
\address{Carlo Casolo, Dipartimento di Matematica e Informatica U. Dini,\newline
Universit\`a degli Studi di Firenze, viale Morgagni 67/a,
50134 Firenze, Italy.}
\email{carlo.casolo@unifi.it}

\author[]{Silvio Dolfi}
\address{Silvio Dolfi, Dipartimento di Matematica e Informatica U. Dini,\newline
Universit\`a degli Studi di Firenze, viale Morgagni 67/a,
50134 Firenze, Italy.}
\email{dolfi@math.unifi.it}

\author[]{Khatoon Khedri}
\address{Khatoon Khedri, Department of Mathematical Sciences, \newline Isfahan University of Technology, 84156-83111 Isfahan, Iran.}
\email{k.khedri@math.iut.ac.ir}

\author[]{Emanuele Pacifici}
\address{Emanuele Pacifici, Dipartimento di Matematica F. Enriques,
\newline Universit\`a degli Studi di Milano, via Saldini 50,
20133 Milano, Italy.}
\email{emanuele.pacifici@unimi.it}

\thanks{The second, third and fifth author are partially supported by the Italian INdAM-GNSAGA}


\begin{abstract}
Let \(G\) be a finite solvable group, and let \(\Delta(G)\) denote the \emph{prime graph} built on the set of degrees of the irreducible complex characters of \(G\). A fundamental result by P.P. P\'alfy asserts that the complement $\bar{\Delta}(G)$ of the  graph \(\Delta(G)\) does not contain any cycle of length \(3\). In this paper we generalize P\'alfy's result, showing that $\bar{\Delta}(G)$ does not contain any cycle of odd length, whence it is a bipartite graph. As an immediate  consequence, the set of vertices of \(\Delta(G)\) can be covered by two subsets, each inducing a complete subgraph. The latter property yields in turn that if \(n\) is the clique number of \(\Delta(G)\), then \(\Delta(G)\) has at most \(2n\) vertices. This confirms a conjecture by Z. Akhlaghi and H.P. Tong-Viet, and provides some evidence for the famous \emph{\(\rho\)-\(\sigma\) conjecture} by B. Huppert.
\end{abstract}

\maketitle

\section{Introduction}
Character Theory is one of the fundamental tools in the theory of finite groups, and, given a finite group $G$, the study of the set $\cd G = \{ \chi (1)\mid \chi \in  \irr G\}$, of all degrees of the irreducible complex characters of $G$, is a particularly intriguing aspect of this theory. One of the methods that have been devised to approach such degree-set is to consider the {\it prime graph} $\Delta (G)$ attached to it. 

The character degree graph $\Delta (G)$ is thus defined as the (simple undirected) graph whose vertex set is the set $\V G$ of all the prime numbers that divide some $\chi (1)\in {\rm cd}(G)$, while a pair $\{p,q\}$ of distinct vertices \(p\) and \(q\), belongs to the edge set $\E G$ if and only if \(p q\) divides an element in \(\cd G\).

There is a well-developed literature on character degree graphs (see for instance the survey \cite{Lew}). A large part of it is focussed on what is obviously one of the natural motivations behind the introduction of such graphs, that is studying to which extent specific properties of a group are reflected by graph theoretical features of its graph, or aimed at describing in detail the degree graph of interesting (classes of) groups.

As regards the investigation about general properties of \(\Delta(G)\), the celebrated Ito-Michler Theorem may be regarded as the first crucial step: this fundamental result characterizes \(\V G\) as the set of all primes \(p\) for which \(G\) does not have an abelian normal Sylow \(p\)-subgroup. 

On the other hand, another fundamental result \emph{in the context of solvable groups} is P\'alfy's ``Three-Vertex Theorem" (\cite{PPP1}): given any three distinct primes in \(\V G\), at least two of them are adjacent in \(\Delta(G)\). For instance, the bound of $3$ for the diameter in the connected case, as well as the structure of the non-connected case as the union of two complete subgraphs, are straightforward consequences of P\'alfy's theorem. A theorem that may  be rephrased by saying that, for a finite solvable group $G$, the complement of $\Delta (G)$ (that we will denote by $\bar\Delta (G)$) does not contain any triangle.  

The main result of this paper is the following extension of P\'alfy's theorem. We stress that Palfy's theorem is actually proved  here with an argument which embeds naturally into our analysis (specifically, in the last two paragraphs of the proof of Theorem~A). This makes our treatment essentially self-contained.
 
\begin{thmA} 
{\sl Let \(G\) be a finite solvable group. Then the graph \(\bar{\Delta}(G)\) does not contain any cycle of odd length. $\bar{\Delta}(G)$ is therefore a bipartite graph.} 
\end{thmA}

An immediate consequence is the following.

\begin{Coro}{\sl Let $G$ be a finite solvable group. Then $\V G$ is covered by two subsets, each inducing a  clique (i.e.\!\! a complete subgraph) in $\Delta (G)$. In particular, for every subset \(\mathcal{S}\) of \(\V G\), at least half of the vertices in \(\mathcal{S}\) are pairwise adjacent in \(\Delta(G)\).}
\end{Coro}

The fact that, for a finite solvable group \(G\), the set \(\V G\) is covered by two subsets each inducing a clique was already known to be true in three special cases: when \(\Delta(G)\) is disconnected (as already mentioned), for metanilpotent groups (\cite[Theorem A]{Lewis2006}) and, in the connected case, when the diameter of \(\Delta(G)\) attains the upper bound \(3\) (see Remark 4.4 in~\cite{CDPS}). By Corollary B, this is indeed a feature of \(\Delta(G)\) in full generality.

Again by Corollary B, if \(G\) is a finite solvable group and \(n\) is the maximum size of a clique in \(\Delta(G)\), then \(|\V G|\leq 2n\). This is precisely what is conjectured by Akhlaghi and Tong-Viet in \cite{ATV} (where the authors prove the result for \(n=3\)). 

Moreover, since the distinct prime divisors of a single irreducible character degree do of course induce a clique in $\Delta (G)$, Corollary B provides some more evidence for the famous (and still open)  \emph{\(\rho\)-\(\sigma\)\! conjecture} by Huppert, which predicts that any finite solvable group $G$ has an irreducible character whose degree is divisible by at least half the primes in $\V G$.

Finally we just recall that, as shown for instance by \(\alt 5\), our theorem (as well as P\'alfy's) does not hold for non-solvable groups.

\section{Notation and preliminary results}

Throughout this paper, every group is tacitly assumed to be finite. As customary, for $n$ a positive integer, $\pi (n)$ denotes the set of all prime divisors of $n$, and, if $G$ is a group,  $\pi (G)= \pi (|G|)$. For a given group \(G\), we denote by \(\Delta(G)\) the character degree graph as defined in the introduction, and write $\bar{\Delta}(G)$ for the complement of \(\Delta(G)\); we emphasize that the set of vertices of $\Delta (G)$ (i.e.\!\! the set of all $p\in \pi (G)$ such that $G$ does not have a normal abelian Sylow $p$-subgroup) is denoted by $\V G$.
We also denote by $\E G$  and $\cE(G)$ the set of edges of $\Delta(G)$ and of $\bar{\Delta}(G)$, respectively.   
If \(N\) is a normal subgroup of \(G\) and \(\lambda\in\irr N\), then \(\irr{G\mid\lambda}\) denotes the set of irreducible characters of \(G\) whose restriction to \(N\) has \(\lambda \) as an irreducible constituent, and \(\cd{G\mid\lambda}\) the set of degrees of the characters in \(\irr{G\mid\lambda}\). 

We will freely use without references some basic  facts of Character Theory  such as  Clifford Correspondence, Gallagher's Theorem, Ito-Michler's Theorem, properties of character extensions and coprime actions (see \cite{Is}).

 Also, we will take into account the following well-known result concerning character degrees (in the following statement, \(\fit G\) and \(\fitd G\) denote the Fitting subgroup and the second Fitting subgroup of \(G\), respectively).

 \begin{lemma}[{\cite[Proposition 17.3]{MW}}]
\label{brodkey} Let $G$ be a solvable group. Then there exists $\chi\in\irr G$
such that $\pi(\fitd G/\fit G) \sbs \pi(\chi(1))$.
\end{lemma}

Let $\Gamma(p^n)$ be the semilinear group on the field $\GF{p^n}$ (see~\cite[Section 2]{MW}). If \(V\) is an \(n\)-dimensional vector space over \({\rm{GF}}(p)\), then \(V\) can be identified with the additive group of a field of order \(p^n\), and in this sense we write \(\Gamma(V)\) for \(\Gamma(p^n)\). 

The following result will be a key ingredient in our discussion. Before stating it, we also recall that a prime $t$ is called a \emph{primitive prime divisor} for  $(a,n)$ (where $a>1$ and $n$ are positive integers) if $t$ divides $a^n-1$ but $t$ does not  divide $a^j-1$ for $1\leq j< n$.

\begin{lemma}\label{KeyLemma}
Let $G=AH$ be a solvable group, where $A=\fit G$ is abelian and $(|A|, |H|) = 1$. Let $q \in \pi(\fit H)$, and
let $s \in \pi(H)$ be distinct vertices such that  $\{q,s\}\not\in\E G$.
Let $Q\in \syl q{\fit H}$,  $S \in \syl sH$ and $L = (QS)^H$ (the normal
closure of \(QS\) in $H$). Then the following conclusions hold.
\begin{enumeratei}
\item \(A=U\times W\), where \(U=[A,Q]=[A,L]\) is an elementary abelian \(p\)-group for a suitable prime \(p\), and \(W=\cent AQ= \cent AL\). 
\item $L = L_0S \leq \Gamma(U)$, where $L_0 = \fit L$ is a cyclic group acting fixed-point freely and irreducibly on $U$.
\item Setting \(|U|=p^n\), there exists a primitive prime divisor \(t\) of \((p,n)\) such that \(t\) divides \(|L_0|\).
\item For every prime \(r\in\pi(H)\) with \(\{r,s\}\not\in\E G\), \(H\) has an abelian normal Sylow \(r\)-subgroup. In particular, $Q \in \syl qH$.
\end{enumeratei}
\end{lemma}

\begin{proof} As concerns (a), (b) and (c), the statement is a special case of Lemma~3.9 in \cite{CDPS}. Thus we only have to prove (d).

Set $K = \cent HU$, and observe that \(U\) is a faithful \(H/K\)-module over \(\GF p\). As clearly $L\cap K = 1$, claims (b) and (c) ensure that \(\fit{LK/K}\) has a characteristic subgroup \(T/K\) of order \(t\), which is therefore characteristic in \(LK/K\) and hence normal in \(H/K\). Since \(t\) is a primitive prime divisor of \((p,n)\), the action of \(T/K\) on \(U\) is easily seen to be irreducible, and therefore \(H/K\) embeds in \(\Gamma(U)\) (see Theorem~2.1 in \cite{MW}, for instance). As a consequence, denoting by \(X/K\) the Fitting subgroup of \(H/K\), the prime divisors of \(H/X\) constitute a clique in \(\Delta(H/K)\), which is a subgraph of \(\Delta(G)\); moreover, Lemma~3.7 of \cite{CDPS} yields that \(X/K\) is cyclic.

By Lemma~\ref{brodkey} and by our assumption of nonadjacency between \(q\) and \(s\), the prime \(s\) is not in \(\pi(\fit H)\). It follows that \(S\) is not normal in \(L\), so \(SK/K\) is not normal in \(LK/K\) and \(H/K\) does not have a normal Sylow \(s\)-subgroup. In other words, \(s\) is a divisor of \(H/X\). Now, taking into account the conclusion of the paragraph above, if \(r\in\pi(H)\) is nonadjacent to \(s\) in \(\Delta(G)\), then \(r\nmid|H/X|\). 
As a consequence, \(H/K\) has a cyclic normal Sylow \(r\)-subgroup, thus \(r\not\in \V{H/K}\). Observe also that \(r\not\in \V K\) as well. In fact, certainly there exists \(\theta\in\irr L\) such that \(\theta(1)\) is divisible by \(s\); thus, if \(\phi\in\irr K\) has a degree divisible by \(r\), then, for any \(\xi\in\irr H\) such that \(\theta\times\phi\in\irr{LK}\) is a constituent of \(\xi_{LK}\), we would have \(r\cdot s\mid\xi(1)\), against the nonadjacency between \(r\) and \(s\) in \(\Delta(G)\). Now, let \(\chi\) be in \(\irr H\), and let \(\beta\) be an irreducible constituent of \(\chi_K\): we have \[\chi(1)=e\cdot\beta(1)\cdot|H:I_H(\beta)|,\] where \(e\) is the degree of an irreducible projective representation of \(I_H(\beta)/K\). If \(r\) divides \(|H:I_H(\beta)|\), then it also divides \(|H:I_H(\theta\times\beta)|\) (where \(\theta\in\irr L\) is as above), and therefore any \(\xi\in\irr{H\mid\theta\times\beta}\) would have a degree divisible by \(r\cdot s\), a
 gain a contradiction. Also, \( e\) is the degree of an irreducible ordinary representation of a Schur covering \(\Gamma\) of \(I_H(\beta)/K\); but, since \(I_H(\beta)/K\) has a cyclic normal Sylow \(r\)-subgroup, we have that \(\Gamma\) has an abelian normal Sylow \(r\)-subgroup, therefore \(r\) does not divide \(e\) as well. Since we observed that \(\beta(1)\) is not divisible by \(r\), we conclude that \(r\nmid\chi(1)\); as this holds for every \(\chi\in\irr H\), we get that \(H\) has an abelian normal Sylow \(r\)-subgroup, as claimed. 
\end{proof}

\section{Proof of Theorem A}

In this section we prove Theorem~A, whereas Corollary~B can be immediately deduced from Theorem A, and its proof is therefore omitted.

The next three lemmas will be the core of our argument. In the following statement, \(\frat G\) denotes as customary the Frattini subgroup of the group \(G\).

\begin{lemma} 
\label{Riduzione}
Let $G$ be a solvable group such that, for every proper factor group $\overline{G}$ of $G$, we have 
$\V{\overline{G}} \neq \V G$.
Let $M$ be a minimal normal subgroup of $G$, let $p \in \V G \setminus \V{G/M}$ and $P\in\syl p G$. Also, let $\pi_p$ be the set of vertices of the connected component
of $p$ in $\bar{\Delta}(G)$. 
Then there exists a normal subgroup $W$ of $G$ such that $\pi_p \subseteq \V{W}$,
and either $\fit W = P$ with $P' = M$, or $\fit W = M$ with $M\cap\frat G=1$; 
in particular, $\fit W$ is complemented in $G$.
\end{lemma}

\begin{proof} 
Let $\mathcal{N}_0$ be the set of the minimal normal subgroups of $G$ which are contained in $\frat G$. 
For $N_0 \in \mathcal{N}_0$, if $p \in \V{G} \setminus \V{G/N_0}$ and $P \in \syl pG$, then
$PN_0/N_0$ is abelian and normal in $G/N_0$, so $P$ is normal in $G$ (as $\fit{G/N_0} = \fit{G}/N_0$) and
$P' = N_0$. Note that $p$ is therefore the unique prime in $\V{G} \setminus \V{G/N_0}$. In this situation, we define $\ww{N}_0= P$; obviously $\ww{N}_0$ is complemented in $G$ and, setting \(F=\fit G\), there exists \(K\nor G\) such that \(F=K\times\ww{N}_0\). 
On the other hand, let  $\mathcal{N}_1$ be the set of the minimal normal subgroups of $G$ that are \emph{not} contained in $\frat G$. If \(N_1\in\mathcal{N}_1$, we define $\ww{N}_1 = N_1$: since  $\ww{N}_1 \cap \frat G = 1$, also in this case we have that $\ww{N}_1$ is complemented in $G$ (see 4.4 in \cite[III]{Hu}) and, as $\ww{N}_1 \leq \zent F$, there exists \(K\nor G\) such that \(F=K\times\ww{N}_1\).   

Observe that $F$ is a direct product of subgroups $\ww N$, where $N$ varies in  $\mathcal{N}_0$ and in a suitable subset of $\mathcal{N}_1$. Moreover, if \(\sigma\) is the set of prime divisors of \(|N_0|\) for \(N_0\in\mathcal{N}_0\), we have \(\oh{{\sigma}'}{F}\cap\frat G=1\), and it is easily seen that \(F\) itself has a complement \(L\) in \(G\). 

Let now $M$ be a minimal normal subgroup of $G$. Take again $p \in \V{G} \setminus \V{G/M}$, $P \in \syl pG$, and let $\ww{M}$ be as above. As mentioned, we can write \(F=K\times \ww M\) where \(K\) is a suitable normal subgroup of \(G\), thus $H = L\ww M$ is a complement for $K$ in $G$. 

Define $W = \cent HK$. We have $W \nor G$,  $W \cap F = \ww M$ and hence $\fit W = \ww M$. Note also that, as $P$ commutes with $K$ modulo $M$, we get $[P, K] \leq M \cap K = 1$; in particular, \(W\) contains the Sylow \(p\)-subgroups of \(H\), and $p \in \V W$.

Let now $q\in \pi_p$ be a vertex of the connected component of $p$ in $\bar{\Delta}(G)$, and 
let $ Q \in \syl qL$. We prove, by induction on the distance $d = d_{\bar{\Delta}(G)}(p, q)$ that 
$1 \neq Q \leq W$. From this it follows immediately $q \in \V W$, since $\fit W = \ww M$.

We consider first the case $d = 1$.
Given a character $\lambda \in \irr{\ww{M}}$ such that $\lambda_M \neq 1_M$, we have 
that $p$ divides $\chi(1)$ for every $\chi \in \irr{G|\lambda}$. 
Observe also that, both in the case $\ww M = P$ as also $\ww M = M$, $\lambda$ extends to 
$I_G(\lambda)$. 
Hence, Gallagher's Theorem and Clifford Correspondence, with  
the nonadjacency between $q$ and $p$ in $\Delta(G)$, imply that $I_{KL}(\lambda)\simeq I_G(\lambda)/\ww M$ contains a 
Sylow $q$-subgroup $Q_0$ of $LK$ as a normal subgroup, and that $Q_0$ is abelian. 
Let now $Q$ be a Sylow $q$-subgroup of $L$; by a suitable choice of $\lambda$, we can assume that
$Q = Q_0 \cap L$. 
Since $K$ centralizes $\ww M$,  $K$ is a normal subgroup of $I_{KL}(\lambda)$ as well.
Hence $Q$ centralizes 
$K = (K \cap Q_0) \times \oh{q'}K$, i.e. $Q \leq W$. 
Observe that  $Q \neq 1$, as otherwise $(K \cap Q_0) \times M$ would contain an abelian normal Sylow 
$q$-subgroup of $G$.

Assume now $d \geq 2$ and let $t$ be the vertex adjacent 
to $q$ in a path of length $d$ from $p$ to $q$ in $\bar{\Delta}(G)$. Let $T\in \syl t L$.
Working by induction on $d$,  we have  $1\ne T\le W$. Observe that, since $t\ne p$ and $\ww M=\fit W$, $T$ is certainly not normal in $W$.
Let $X = T^G=T^W$ (the normal closure of $T$ in $W$)  and $U = \Oh tX$; then $1\ne U = [U,T]$. 
Let $\overline U = U/V$ be a chief factor of $G$. 
If $\overline U\leq\frat{\overline G}$, then $\overline{T}$ is normal in $\overline G$ and so 
$[\overline U,\overline T]=1$; in particular $[U, T]\leq V<U$, a contradiction. 
Thus $\overline U \cap \frat{\overline G} = 1$ and hence  $\overline U$ is complemented in $\overline G$.
Now, for $1\ne \lambda\in \irr {\overline U}$, the prime $t$ divides the index 
$|\overline G: I_{\overline G}(\lambda)|$ and hence the fact that  $t$ is not adjacent to $q$ in 
$\Delta(\overline G)$ forces $I_{\overline G}(\lambda)$ to contain $\overline Q_0 = Q_0V/V$ as an abelian 
normal subgroup, where $Q_0$ is a Sylow $q$-subgroup of $G$. 
Note that, by a suitable choice of $\lambda$, we can assume  $Q = Q_0 \cap L$.
 
Now, let $Y = KU$; since $K\cap U = 1$, we have $\overline Y = Y/V = \overline K\times \overline U$, where $\overline K = KV/V\simeq K$. Let $1\ne \lambda\in \irr {\overline U}$ and $Q_0$ be as above. For every $\phi \in \irr{\overline K}$ we have  $I_{\overline G}(\phi\times\lambda)\le I_{\overline G}(\lambda)$, hence $t$ divides the index $|\overline G: I_{\overline G}(\phi\times\lambda)|$; since $t$ and $q$ are not adjacent in $\Delta(\overline G)$, this yields $\overline Q_0 \le I_{\overline G}(\phi\times\lambda) \le I_{\overline G}(\phi)$, and this holds for every $\phi\in \irr {\overline K}$. Thus $[\overline K,\overline Q_0]=1$, that is $[K, Q_0]\le K\cap V= 1$. Hence, in particular, $Q\le \cent H K=W$. 
Moreover,  $Q_0 \cap K$ is abelian; so  $Q$ cannot be trivial, for otherwise 
$Q_0 = Q_0 \cap F$ would be  abelian and normal in  $G$. The proof is complete. 
\end{proof}

\begin{lemma} Let \(p\) be a prime, \(E\) an elementary abelian \(p\)-group, and \(H\) a \(p'\)-group acting faithfully on \(E\). Let \(G=EH\), and write $\pi_0 = \pi (\fit H)$. 
\begin{enumeratei}
\item Let $q$, $r$, $s$ be distinct prime divisors of $|H|$ such that $\{q, s\},\{s,r\} \in \cE(G)$, and $q\in \pi_0$. Then $s\not\in \pi_0$ and $r \in \pi_0$. Moreover, \(H\) has both a  normal abelian Sylow \(q\)-subgroup and Sylow \(r\)-subgroup.
\item Let $p_1, p_2, p_3, p_4$ be distinct prime divisors of $|H|$ such that $\{ p_i, p_{i+1}\} \in \cE(G)$ for \(i\in\{1,2,3\}\), and $p_1\in \pi_0$. Then $\{p_1, p_4\} \in \cE(G)$.
\end{enumeratei}
\label{alternanza}
\end{lemma}

\begin{proof} As it regards (a), by Lemma~\ref{brodkey} the subgraph of $\Delta(G)$ induced by $\pi_0$ is complete, so the claim $s\not\in \pi_0$ is immediate from the assumption of nonadjacency between $q$ and $s$. The remaining claims of (a) follow from Lemma~\ref{KeyLemma}, therefore we only have to prove (b).


Let $p_1, p_2, p_3, p_4$ be as in the assumptions. Then (a) yields $\{p_1, p_3\}\subseteq \pi_0$, whereas $\{p_2, p_4\}\cap \pi_0=\emptyset$; actually, setting $P_i \in \syl {{p_i}}H$ for $i = 1, 2, 3, 4$, we have that \(P_1\) and \(P_3\) are abelian and normal in \(H\).

Write $L_1 = (P_1P_2)^H$, $L_3 = (P_3P_4)^H$, $L = L_1L_3$, $M=\cent EL$ and $V = E/M$. Also, let $U_1 = [ V, P_1]$ and $U_3 = [V, P_3]$. Then, clearly,  $V = U_1U_3$, and, again by Lemma~\ref{KeyLemma}, we have that $U_1 = [V, L_1]$ and $\ U_3 = [V,L_3]$ are irreducible $L$-modules, hence irreducible $H$-modules.

Let us assume $U_1\ne U_3$; then $V = U_1\times U_3$, $U_1 = \cent V{L_3}$ and $U_3 = \cent V{L_1}$. Thus we get $L_1\cap L_3 =1$, whence $L = L_1\times L_3$ and  $VL = U_1L_1\times U_3L_3$. As $p_2$ and $p_3$ divide, respectively, the degree of a character of $U_1L_1$ and of $U_3L_3$, we get that \(p_2\) and \(p_3\) are adjacent in the graph $\Delta (VL)$, which is a contradiction because $VL$ is isomorphic to a section of $G$.

Thus we have $U_1=U_3$. Setting \(U=[E,P_1]\), we now have $E = U\times M$.
So, $[E,P_1]= [E,P_3]$ and $p_1\cdot p_3$ divides $|H:I_H(\lambda)|$ for every nonprincipal $\lambda \in {\rm Irr}(U)$. 
Observe also that $U$ is complemented in $G$ by $MH$. 
Let \(\chi\) be an irreducible character of \(G\) whose degree is divisible by \(p_1\). Since $P_1$ is an abelian normal Sylow subgroup of $H$, then $MP_1$ is an abelian normal  subgroup of $MH\simeq G/U$, and hence the kernel of \(\chi\) cannot contain \(U\). Therefore, denoting by \(\lambda\) a (nonprincipal) irreducible constituent of \(\chi_U\), the degree of \(\chi\) is divisible by \(|G:I_G(\lambda)|=|H:I_H(\lambda)|\), thus by \(p_1\cdot p_3\). As \(p_4\) is not adjacent to \(p_3\) in \(\Delta(G)\), we conclude that \(p_4\) does not divide \(\chi(1)\). But this holds for every \(\chi\in\irr G\) such that \(p_1\mid\chi(1)\), hence $\{p_1, p_4\} \in \cE(G)$, as desired.
\end{proof}

\begin{lemma}
Let \(G\) be a solvable group, and \(p\) a prime in \(\V G\). Assume that \(P=\fit G\) is a  Sylow \(p\)-subgroup of \(G\) with \(P' \neq 1\) minimal normal in \(G\) and, denoting by \(H\) a complement for \(P\) in \(G\), define \(\pi_0=\pi(\fit H)\). Assume further that \(p_0=p\), \(p_1\), \(p_2\) and \(p_3\) are distinct primes in \(\V G\) (except possibly \(p_0=p_3\)) such that, for \(i\in\{0,1,2\}\), 
$\{p_i, p_{i+1}\} \in \cE(G)$. If \(\{p_1,p_2\}\cap\pi_0=\emptyset\), then \(p \neq p_3\) and 
$\{p, p_3\} \in \cE(G)$.
\label{pPartenza}
\end{lemma}

\begin{proof}
Assuming that \(p_1\) is not in \(\pi_0\), and that either \(p=p_3\) or \(p\) is adjacent to \(p_3\) in \(\Delta(G)\), we will show  \(p_2\in\pi_0\).

Setting \(M=P'\), let \(\widehat{M}\) denote the dual group \(\irr M\), and let \(\lambda\) be a nontrivial element in \(\widehat{M}\). For any \(\tau\in\irr{P\mid\lambda}\) we get \(I_H(\tau)\leq I_H(\lambda)\), because \(M\leq\zent P\) and so \(\tau_M\) is a multiple of \(\lambda\); but, by \cite[13.28]{Is}, there exists \(\theta\in\irr{P\mid\lambda}\) for which in fact equality holds. Now, \[\cd{G\mid\theta}=\{|H:I_H(\lambda)|\cdot\theta(1)\cdot\xi(1)\;:\;\xi\in\irr{I_G(\lambda)/P}\},\] thus, the nonadjacency between \(p\) and \(p_1\) forces \(I_H(\lambda)\simeq I_G(\lambda)/P\) to contain a unique abelian Sylow \(p_1\)-subgroup of \(H\). 
Set \(K=\cent H M\) and   note  that, for \(\lambda\in\widehat{M}\setminus\{1_M\}\), \(K\) is a normal subgroup of \(I_H(\lambda)\). Denoting by \(P_1\) the unique Sylow \(p_1\)-subgroup of \(H\) contained in \(I_H(\lambda)\), we have that \(P_1\cap K\) is a normal Sylow \(p_1\)-subgroup of \(K\), whence it lies in \(\fit H\); but we are assuming \(p_1\not\in\pi_0\), thus \(p_1\nmid |K|\) and \(P_1\leq\cent H K\). As then $p_1$ divides $|H/K|$, an application of Lemma 3.6 in \cite{CDPS} yields that \(H/K\) can be identified with a subgroup of the semilinear group \(\Gamma(\widehat{M})\).  Finally, setting \(X/K=\fit{H/K}\), the prime \(p_1\) divides the order of the cyclic group \(|H/X|\); in fact, if we assume the contrary, then \(P_1K=P_1\times K\) is normal in \(H\), and therefore \(P_1\leq\fit H\) against our assumptions. Observe that the prime divisors of \(|H/X|\) constitute a clique in \(\Delta(G)\); in particular, \(p_2\nmid|H/X|\).

Assume now that \(p_2\) is a divisor of \(|X/K|\). Hence, if \(P_2\) is a Sylow \(p_2\)-subgroup of \(H\), we get \(K<P_2K\unlhd H\) and \([M,P_2K]=M\). It follows that, for every \(\lambda\in\widehat M\setminus\{1_M\}\), \(p_2\) divides \(|H:I_H(\lambda)|\). In particular, \(p_2\) divides the degree of every irreducible character of \(G\) whose kernel does not contain \(M\). Now, there exists \(\chi\in\irr G\) such that \(\{p,p_3\}\subseteq\pi(\chi(1))\), and the kernel of such a \(\chi\) clearly does not contain \(M\). As a consequence, \(p_2\cdot p_3\) divides \(\chi(1)\), a contradiction. Our conclusion so far is that \(P_2\leq K\).

Consider now the normal closure \(P_1^H\) of \(P_1\) in \(H\), set \(Z=\zent K\) and \(L=P_1^HZ\unlhd H\). Recalling that \(P_1\leq\cent H K\unlhd H\), we get \(L\leq\cent H K\) as well, and so \(K\cap L=Z\).   In other words, \(KL/Z=(K/Z)\times(L/Z)\). Observe that \(p_1\) lies in \(\V{L/Z}\), as otherwise we would have \(P_1Z=P_1\times Z\unlhd L\), whence \(P_1\leq\fit H\) against our assumptions. Now, the nonadjacency between \(p_1\) and \(p_2\) in \(\Delta(G)\) forces \(p_2\not\in \V{K/Z}\). Therefore we get \(P_2Z/Z\unlhd H/Z\). But \(P_2Z\) is nilpotent, and we conclude that \(P_2\leq \fit H\), as desired. 
\end{proof}

We are now ready to prove Theorem A, that we state again.

\begin{theorem} Let \(G\) be a solvable group. Then the graph \(\bar{\Delta}(G)\) does not have any cycle of odd length. 
\end{theorem}

\begin{proof}
Let \(G\) be a counterexample of minimal order to the statement, and let \(\ell\) be the smallest odd number for which a cycle \(\mathcal{C}\) of length \(\ell\) can be found in \(\bar{\Delta}(G)\). 
Let \(\{p_0,p_1,...,p_{\ell-1}\}\) be the set of distinct vertices lying in \(\mathcal{C}\), where we assume $\{p_{\ell-1}, p_0\}, \{p_i, p_{i+1}\} \in \cE(G)$ for every \(i\in\{0,...,\ell-2\}\); 
of course, by our choice of \(\ell\), these are the only edges in  $\bar{\Delta}(G)$ between vertices in \(\mathcal{C}\).

We start by observing that, for every proper factor group $\overline{G}$  of \(G\), some vertex of \(\mathcal{C}\) does not belong to  \( \V{\overline{G}}\), for otherwise \(\mathcal{C}\) would be a cycle of odd length in \(\bar{\Delta}(\overline{G})\), against the minimality of \(G\). 
Let $M$ be a minimal normal subgroup of $G$ and let $p = p_0 \not\in \V{G/M}$. 

By Lemma~\ref{Riduzione} and the minimality of $G$, either \(F=\fit G\) is a nonabelian Sylow $p$-subgroup of \(G\) with \(F'\) minimal normal in \(G\), or \(F\) is itself minimal normal in \(G\). In any case, \(F\) has a complement \(H\) in \(G\); in what follows, we will denote by \(\pi_0\) the set of prime divisors of \(\fit H\).

Let us first consider the case when \(F\) is a minimal normal subgroup of \(G\). 
We clearly have \(\{p_0,...,p_{\ell-1}\} \subseteq\pi(H)\). Also, if \(P\) is the abelian normal Sylow \(p\)-subgroup of \(H\), we get \(\cent F P=1\). In other words, \(p\) divides \(|H:I_H(\lambda)|\) for every \(\lambda\in\irr F\setminus\{1_F\}\). As a consequence, we have that \(I_H(\lambda)\) contains a unique Sylow \(p_1\)-subgroup of \(H\) for every \(\lambda\in\irr F\setminus\{1_F\}\). Now, by \cite[Lemma~3.6]{CDPS}, \(H\) can be identified with a group of semilinear maps on the dual group of \(F\) 
(note that, as $|\pi(H)| > 2$, $H \not\leq \rm{GL}_2(3)$). 
In particular, \(H/\fit H\) is abelian, and therefore both \(\pi_0\) and \(\pi(H)\setminus \pi_0\) induce complete subgraphs in \(\Delta(G)\). It follows that two consecutive vertices of the cycle \(\mathcal C\) cannot lie both in \(\pi_0\) nor both in \(\pi(H)\setminus\pi_
 0\), and this is clearly impossible because \(\ell\) is odd.

Our conclusion so far is that \(F\) is a nonabelian Sylow \(p\)-subgroup of \(G\), hence we are in the setting of Lemma~\ref{pPartenza}. Let us assume first \(\ell\geq 5\). Since $\{p, p_3\} \not\in \cE(G)$, we conclude that one among \(p_1\) and \(p_2\) lies in \(\pi_0\). Now, adopting the bar convention for \(\overline G=G/\frat G\), we have that \(\overline F\) is an elementary abelian \(p\)-group, and \({\overline H}\simeq H\) is a \(p'\)-group acting faithfully (by conjugation) on \(\overline F\). Therefore, we can apply Lemma~\ref{alternanza}. Observe that \(p_1,...,p_{\ell-1}\) all lie in \(\V{\overline G}\cap \pi(\overline H)\), and they are of course consecutive vertices of a path in \(\bar{\Delta}(\overline G)\). If \(p_1\in\pi_0\), then Lemma~\ref{alternanza}(b) yields that 
$\{p_1, p_4\}  \in \cE(G)$, a contradiction. On the other hand, if \(p_1\) does not lie in \(\pi_0\), then we have seen that \(p_2\)
  does, and so does \(p_{\ell-1}\) by an iterated application of Lemma~\ref{alternanza}(a). But again Lemma~\ref{alternanza}(b) yields now  $\{p_{\ell-4}, p_{\ell-1}\} \in \cE(G)$, the final contradiction for the case \(\ell\geq 5\).

It remains to treat the situation when \(\ell=3\). In view of Lemma~\ref{pPartenza}, we may assume that \(p_1\) is a divisor of $|\fit H|$ (thus \(p_2\) is not). Set $P_1=\oh{p_1}H$, \(P_2\in{\rm Syl}_{p_2}(H)\), and define \(L=(P_1P_2)^H\). We first observe that $P_1$ centralizes $F'$ and, for every \(\lambda\in\irr{F'}\setminus\{1_{F'}\}\), a unique \(H\)-conjugate of \(P_2\) lies in \(I_H(\lambda)\). In fact, arguing as in the second paragraph of the proof of Lemma~\ref{pPartenza}, we see that \(|H:I_H(\lambda)|\) is coprime with \(p_1\cdot p_2\) and \(p_2\not\in \V{I_H(\lambda)}\), for every \(\lambda\in\irr{F'}\setminus\{1_{F'}\}\). Note also that, for similar reasons, \(P_1\) fixes all the nonlinear irreducible characters of \(F\); therefore, setting \(U=[F,P_1]\), Theorem~19.3(a) in \cite{MW} yields \(U'=F'\). Now, by Lemma~\ref{KeyLemma}, we get \(U=[F,L]\), and \(U/F'\) is a faithful irreducible \(L\)-module. Also, since \(U=\fit{UL}\) is nonabelian, we get  \(\{p_0
 ,p_1,p_2\}\subseteq \V{UL}\); but \(UL\) is normal in \(G\), thus the minimality of \(G\) forces \(F=U\) (so \(F'=\cent F {P_1}\)) and \(L=H\). Finally, observe that no nonprincipal irreducible character of \(F/F'=[F/F',P_1]\) is fixed by \(P_1\); therefore, in view of the nonadjacency between \(p_1\) and \(p_2\) in \(\Delta(G)\), for every nonprincipal \(\mu\in\irr{F/F'}\), a unique \(H\)-conjugate of \(P_2\) lies in \(I_H(\mu)\). 

We are then in a position to apply Lemma~3.8 in \cite{CDPS}, obtaining that $|F/F'| = |F'|$. Let \(x\) be an element in \(F\setminus F'\), and consider the map \(f\mapsto[f,x]\) from \(F\) to \(F'\); recalling that \(F'\leq\zent F\), this defines a group homomorphism whose kernel strictily contains \(F'\). As a consequence, \(|[F,x]|\) is strictly smaller than \(|F/F'|=|F'|\) and we conclude that $[F, x]$ is properly contained in $F'$. Let $K$ be a maximal 
subgroup of $F'$ containing $[F, x]$, so $K \unlhd FP_1$, and
take $\lambda \in \irr {F'}$ with $K = \ker{\lambda}$. Since \(P_1\) fixes no nontrivial element in \(F/F'\), we can choose an element $y \in P_1$ such that, setting  $z = [x,y]$, we have $z \not \in K$.  
As $\ker{\lambda^F} = K$, there exists $\theta \in \irr{F\mid\lambda}$ such that
$z \not\in \ker{\theta}$. 
As already mentioned, $\theta$ is $P_1$-invariant, and hence it extends to a character
$\psi \in \irr{FP_1}$. 
If $\Psi$ is a representation affording $\psi$, then $\Psi(x)$ is a scalar matrix; in fact, $\Psi_F$ affords $\theta$, and $x \in \zent{\theta}$ because $xK \in \zent{F/K}$. In particular, $\Psi(x)$ commutes with $\Psi(y)$. But now $\Psi(z)$ is the identity matrix and  $z$ lies in \(\ker{\psi_F} = \ker{\theta}$, the final contradiction that completes the proof. 
\end{proof}

\end{document}